\documentclass[12pt]{article}
\usepackage{amsmath}
\usepackage{amsfonts}
\usepackage{amssymb}
\usepackage{amsthm}
\usepackage{stackrel}

\usepackage{color}

\usepackage{figsize}

\usepackage{graphics}

\newtheorem{ex}{Example}[section]

\newtheorem{lem}{Lemma}[section]
\newtheorem{df}{Definition}[section]
\newtheorem{thm}{Theorem}[section]

\newcommand{\go}[1]{\mathfrak{#1}}
\def\binom#1#2{{#1}\choose{#2}}
\newcommand{\R}{{\rm I}\kern-0.18em{\rm R}}
\newcommand{\1}{{\rm 1}\kern-0.25em{\rm I}}
\newcommand{\E}{{\rm I}\kern-0.18em{\rm E}}
\newcommand{\p}{{\rm I}\kern-0.18em{\rm P}}
\makeatletter

\def\@fnsymbol#1{\ensuremath{\ifcase#1\or a\or b\or c\or d\or \e\or f\or *\dagger 	\or \ddagger\ddagger \else\@ctrerr\fi}}

\title{A new definition of random sets}
\author{Vesna Gotovac\footnote{University of Split}, Kate\v{r}ina Helisova \footnote{Czech Technical University in Prague},\\ Lev B. Klebanov\footnote{Charles University} and Irina V. Volchenkova\footnotemark[2]\, \footnote{The order of authors is alphabetical and has no other sense}}
\date{}
\begin{document}
\maketitle

\begin{abstract} 
A new definition of random sets is proposed. It is based on the distance in measurable space and uses negative definite kernels for continuation from initial space to that of random sets. This approach has no connection to Hausdorff distance between sets. 

\end{abstract}

\section{Introduction}\label{sec1} 
\setcounter{equation}{0}

Mathematical theory of random sets is very popular nowadays. It has also many applications in physics, biology, medicine and other field of science. Let us mention the books \cite{Ma} and \cite{Mo} as providing main notions and basic result in the theory. 

Main approach to definition of random set is as following. One has a metric space. On the set of its compact subsets implemented with Hausdorff metric are defined Borel probability measures, which are the ``distributions of random sets". To have more rich structure one considers compact convex subsets of Euclidean space $\R^d$ with Minkowski sum as operation.

In this paper we would like to define random events, that are random elements of a Boolean algebra (B.A.) with a finite (or just probabilistic) measure on it. Of course, such events may be considered as subsets from a realization of B.A. as clopen subsets algebra of totally disconnected compact space (see, for example \cite{Vl}). The possibility of such realization is guaranteed by well-known Stone theorem. However, corresponding totally disconnected compact space is not metrizable in typical situations. Therefore, we cannot use standard approach based on Hausdorff distance. 
Our approach starts from normed B.A., i.e. complete B.A. endowed with positive finite measure\footnote{Measure is a non-negative countable additive function on B.A.}. Without loss of generality we can suppose that this is probabilistic measure, or, shortly, probability. It allows one to introduce a distance on B.A., which will be used to define class of Borel subsets of B.A. (see, for example, \cite{Vl}).

Assume we are given a measure space $\{\Omega, {\mathcal E}, m\}$. To this space there corresponds the ``metric structure", i.e., the Boolean algebra $\tilde{\mathcal E}$ resulting
from factorization of initial $\sigma$-algebra  by the ideal of negligible sets. This
metric structure is a complete B.A. endowed with the measure $m$. We will apply the term "metric structure" to the normed B.A. $\{\tilde{\mathcal E}, \tilde{m}\}$ as well. The inverse is also true. Namely, {\it for each normed B.A. $\{{\mathcal X},\mu\}$ there exists a measure space $\{\Omega, {\mathcal E}, m\}$ such that the normed B.A.s $\{{\mathcal X},\mu\}$ and $\{\tilde{\mathcal E}, \tilde{m}\}$ are isomorphic}. In other words, each normed B.A. is isomorphic to some metric structure (see \cite{Vl}). According to this result we will further consider normed B.A.s identifying them with isomorphic metric structures.

Let us now consider complete normed B.A. $\{{\mathcal X},\mu\}$ and let measure space $\{\Omega, {\mathcal E}, m\}$ be such that $\{\tilde{\mathcal E}, \tilde{m}\}$ is isomorphic to $\{{\mathcal X},\mu\}$. Define
\begin{equation}\label{eq1} 
{\mathcal L}(A,B) = \mu (A \Delta B), \; \; A,B \in {\mathcal X}.
\end{equation}

\begin{df}\label{de0}
Let $L(x,y)=L(y,x)$ be a real function given on an abstract set $\mathcal{R}$. We say $L$ is negative definite kernel if for any positive integer $n$, any points $x_1, \ldots ,x_n$ from $\mathcal{R}$ and any real constants $c_1, \ldots ,c_n$ under condition $\sum_{k=1}^{n}c_k =0$ the following inequality holds
\[ \sum_{i=1}^{n}\sum_{j=1}^{n}L(x_i,x_j)c_ic_j \leq 0.\]
\end{df}

\begin{lem}
The function $\mathcal L$ given by (\ref{eq1}) is a negative definite kernel on $\mathcal X^2$.
\end{lem}
\begin{proof} 
We have
\[{\mathcal L}(A,B) = \mu (A \Delta B) = \int_{\Omega} \bigl(\1_A(x) + \1_B(x) -2 \1_A(x)\1_B(x)\bigr)dm. \]
Suppose that $c_1, \ldots ,c_n$ are some constant under condition $\sum_{k=1}^{n}c_k =0$, and $A_1, \ldots ,A_n$ are elements of $\mathcal X$. Then
\[\sum_{j=1}^{n}\sum_{k=1}^{n}{\mathcal L}(A_j,A_k)c_j c_k = \] \[=\sum_{j=1}^{n}\sum_{k=1}^{n}\int_{\Omega} \bigl(\1_{A_j}(x) + \1_{A_k}(x) -2 \1_{A_j}(x)\1_{A_k}(x)\bigr)dm\cdot c_jc_k =\]
\[ = -2\int_{\Omega} \Bigl( \sum_{k=1}^{n}\1_{A_k}c_k\Bigr)^2 dm \leq 0. \]
\end{proof}

It is easy to see that 
\begin{eqnarray}\label{eq2}
{\mathcal L}(A,B) = \mu (A \Delta B) = \int_{\Omega} \bigl(\1_A(x) - \1_B(x) \bigr)^2 dm \\\nonumber =\int_{\Omega} \bigl|\1_A(x) - \1_B(x) \bigr|^{\alpha} dm
\end{eqnarray}
for any $\alpha>0$. Therefore, $\Bigl(\mathcal{L}(A,B)\Bigr)^{\min(1,1/\alpha)}$ is a metric on the B.A. $\mathcal X$. For $\alpha=2$ the distance is equivalent to $L^2(\Omega,m)$ -- distance. B.A. $\mathcal{X}$ with this distance possesses an isometry on a subset of a Hilbert space. In this paper we concentrate ourselves on $L^2$ case. 

Let us mention that there are many topologies introduced in B.A.s. The most popular between them is order topology ($(o)$-topology). It is known that {\it the topology of the metric space $\{\mathcal{X},{\mathcal L}\}$ coincides with the $(o)$-topology} (for the definition of $(o)$-topology and mentioned result see \cite{Vl}). Therefore, $(o)$-topology coincides with that generated by ${\mathcal L}^{1/2}$ and, consequently, with topology induced from Hilbert space. Further we do not use $(o)$-topology and, therefore, omit the details.  

Negative definite kernel $\mathcal L$ generates positive definite kernel $\mathcal K$ (see \cite{Kl}) as follows:
\begin{equation}\label{eq3}
{\mathcal K}(A,B) = \frac{1}{2} \Bigl({\mathcal L}(A,\Omega)+{\mathcal L}(\Omega,B) - {\mathcal L}(A,B)\Bigr) = \mu(A\cap B).
\end{equation}
The kernel $\mathcal K$ plays role of inner product in corresponding Hilbert space \cite{Kl}. However, it is obvious from representation
\[ {\mathcal K}(A,B) = \int_{\Omega} \1_A(x)\cdot \1_B(x) dm. \] 
This allows us to define some characteristics of elements of our B.A. using geometric properties of Hilbert space. For example, we may define norm of a set as $\|A\|={\mathcal K}(A,A)= \mu (A)$ and an angle $\alpha (A,B)$ between sets (elements of B.A.) $A$ and $B$ by setting
\[\cos \alpha (A,B) = \frac{\int_{\Omega} \1_A(x)\cdot \1_B(x) dm}{\Bigl(\int_{\Omega} \1_A^2(x)dm\Bigr)^{1/2}\Bigl(\int_{\Omega} \1_B^2(x)dm\Bigr)^{1/2}} =\frac{\mu(A\cap B)}{(\mu (A)\cdot \mu (B))^{1/2}}.\]
It is clear that\footnote{Let us reminde that the measure $\mu$ is strictly positive} 
\[ 0 \leq \cos \alpha (A,B) \leq 1,\] 
\[\cos \alpha (A,B) =0 \Longleftrightarrow A\cap B = \emptyset,\] 
\[ \cos \alpha (A,B) = 1 \Longleftrightarrow A=B. \]
Therefore, the sets with empty intersection may be considered as orthogonal. The main property of the measure $\mu$: $\mu(A\cup B) = \mu (A)+ \mu(B)$ for orthogonal $A$ and $B$ may be interpreted as Pythagorean theorem in Hilbert space. 

Let $A_1, \ldots ,A_n$ be complete system of events,i.e. $n$ elements of B.A. satisfying to the conditions $\bigcup_{j=1}^{n}A_j=\Omega$ and $A_i \cap A_j = \emptyset$ for $i \neq j$. Then $\sum_{j=1}^{n}\1_{A_j} = \1_{\Omega}=1$ and $\int_{\Omega}\1_{A_i}\cdot \1_{A_j}dm = 0$ for $i \neq j$, i.e. the functions $\1_{A_1}, \ldots, \1_{A_n}$ compose an orthogonal system. Inverse statement is not completely true. If some indicator-functions $\1_{A_1}, \ldots, \1_{A_n}$ compose an orthogonal system then $A_i \cap A_j = \emptyset$ for $i \neq j$, but, possible, $\bigcup_{j=1}^{n}A_j \neq\Omega$. 

Suppose now that indicator-functions $\1_{A_1}, \ldots, \1_{A_n}$ compose an orthogonal system, and $\1_A$ is indicator, corresponding to an event $A$. We can find the best approximation of $\1_A$ by linear combinations of $\1_{A_1}, \ldots, \1_{A_n}$ in our Hilbert space:
\[ \min_{a_1, \ldots ,a_n}\int_{\Omega}\Bigl(\1_A(x)-\sum_{j=1}^{n}a_j\1_{A_j}\Bigr)^2 dm .\] 
It is well-known (and easy to find) that optimal values of coefficients $a_j$ are
\begin{equation}\label{eqA}
 a_j^{*}=\frac{\int_{\Omega}\1_A(x)\cdot \1_{A_j}dm}{\int_{\Omega}\1_{A_j}^2dm}= \frac{\mu{(A \cap A_j)}}{\mu(A_j)}, \;\;\; j=1, \ldots ,n.
\end{equation}
Obviously, optimal coefficients $a_j^*$ are conditional probabilities of $A$ given $A_j$, and their interpretation is clear. However, how is it possible to give interpretation of 
$\sum_{j=1}^{n}a_j^*\cdot \1_{A_j}$? This sum is not indicator-function and, therefore, does not correspond to any event. Let us mention that $0 \leq a_j^* \leq 1$ and
$\sum_{j=1}^n a_j^* \leq 1$. If we define $a_{n=1}^* = 1-\sum_{j=1}^n a_j^*$ and $A_{n+1}=\emptyset$ then \[\sum_{j=1}^{n}a_j^*\cdot \1_{A_j}= \sum_{j=1}^{n+1}a_j^*\cdot \1_{A_j}. \]
 Here the sum $\sum_{j=1}^{n+1}a_j^*\cdot \1_{A_j}$ is the convex combination of indicators $\1_{A_j}$, $j=1, \ldots ,n+1$ and, therefore, may be interpreted as the mean value of random indicator $\1_{\mathcal A}$ (or, equivalently, random set ${\mathcal A}$), where ${\mathcal A}$ takes values $A_{j}$ with the probabilities $a_j^*$, $j =1,\ldots ,n+1$. This shows us a very natural approach to the notion of random sets. 
 
\section{Definition of random sets}\label{sec2}
\setcounter{equation}{0}

We can consider a B.A. $\mathcal X$ as a metric space with the distance 
\begin{equation}\label{eq4} 
d(A,B) =d_2(A,B) = \mathcal{L}^{1/2}(A,B) =\bigl(\mu (A\Delta B)\bigr)^{1/2}
\end{equation}
for $A,B \in {\mathcal X}$. Denote by $\go A$ the Borel $\sigma$-algebra of subsets $\mathcal X$. Each element of $\go A$ is a class of elements from $\mathcal X$, say a family $\{A_{\xi}, \xi \in \Xi\}$, where $A_{\xi} \in {\mathcal X}$. As was mentioned above, we may identify each $A_{\xi}$ with corresponding indicator $\1_{A_{\xi}}(x)$. Any probability measure (not necessary strictly positive) on $\go A$ is considered as a distribution of random set. Our first aim is to define a distance between such measures. Let us remind the definition of strongly negative definite kernel \cite{Kl}. 

\begin{df}\label{de1}
Let $\go X$ be an abstract set, and $\{{\go B}, Q\}$ is a $\sigma$-algebra of its subsets with given probability measure $Q$. Suppose that ${\go L}(X,Y)$ be a real function on $\go X^2$ such that ${\go L}(X,Y)={\go L}(Y,X)$. Let $h(x)$ be a function on $\go X$ integrable with respect to $Q$ under condition
\[ \int_{\go X} h(x) dQ(x) = 0. \] 
We shall say that $\go L$ is strongly negative definite kernel if it is negative definite and equality
\[\int_{\go X}\int_{\go X}{\go L}(x,y) h(x)h(y)dQ(x) dQ(y) = 0 \]
implies that $h(x)=0$ $Q$-almost everywhere for any measure $Q$.
\end{df}

Let us come back to our model $\{{\mathcal X},\go A \}$. Denote by $\go P$ the set of all probabilities on $\go A$. Let $\go m$ be a probability on $\go A$. As it was mentioned above, we consider ${\go m} \in {\go P}$ as a distribution of corresponding random set. We would like to define a distance on a space of some characteristics of ${\go m}\in\go P$, which is proposed as a distance between corresponding independent random sets. To this aim introduce a negative definite kernel on the pairs of such characteristics. Namely, define
\begin{eqnarray}\label{eq5}
\mathcal{N}(\go m,\go n)=2 \int_{\mathcal X}\int_{\mathcal X}{\mathcal L}(A,B)d{\go m}(A)d{\go n}(B) -\\ \nonumber
- \int_{\mathcal X}\int_{\mathcal X}{\mathcal L}(A,B)d{\go m}(A)d{\go m}(B) -\int_{\mathcal X}\int_{\mathcal X}{\mathcal L}(A,B)d{\go n}(A)d{\go n}(B).
\end{eqnarray} 
Let us transform expression (\ref{eq5}) using (\ref{eq4}). 
\[\mathcal{N}(\go m,\go n)=2 \int_{\mathcal X}\int_{\mathcal X}\mu(A\Delta B)d{\go m}(A)d{\go n}(B) -\]  
\[-\int_{\mathcal X}\int_{\mathcal X}\mu(A\Delta B)d{\go m}(A)d{\go m}(B) -\int_{\mathcal X}\int_{\mathcal X}\mu(A\Delta B)d{\go n}(A)d{\go n}(B) = \]
\[= \int_{\Omega}\Bigl(2 \int_{\mathcal X}\int_{\mathcal X}\bigl(\1_A(x)+\1_B(x) - 2 \1_A(x)\cdot \1_B(x)\bigr)d{\go m}(A)d{\go n}(B)-\]
\[-\int_{\mathcal X}\int_{\mathcal X}\bigl(\1_A(x)+\1_B(x) - 2 \1_A(x)\cdot \1_B(x)\bigr)d{\go m}(A)d{\go m}(B)-\] \[-\int_{\mathcal X}\int_{\mathcal X}\bigl(\1_A(x)+\1_B(x) - 2 \1_A(x)\cdot \1_B(x)\bigr)d{\go n}(A)d{\go n}(B) \Bigr)dm(x)=\]
\[= \int_{\Omega}\Bigl(f_{\go m}(x)-f_{\go n}(x)\Bigr)^2 dm(x),\]
where 
\[f_{\go m}(x) =\int_{\mathcal X}\1_A(x)d{\go m}(A), \;\; f_{\go n}(x) =\int_{\mathcal X}\1_A(x)d{\go n}(A). \]
Now we see that ${\mathcal N}(\go m,\go n)$ is a negative definite kernel on the set $\go P^2$, which is strongly negative definite kernel on the space $\go F^2$ of pairs $(f_{\go m}(x),f_{\go n}(x))$ for $(\go m, \go n) \in \go P^2$. The set $\go F$ is a set of functions $\{f_{\go m},\; {\go m}\in {\go P}\}$. Finally, we define a distance ${\go N}$ on $\go F$:
\begin{equation}\label{eq6}
{\go N}(\go m^{\prime}, \go n^{\prime}) = \Bigl(\int_{\Omega}\bigl(f_{\go m}(x)-f_{\go n}(x)\bigr)^2 dm(x)\Bigr)^{1/2},
\end{equation}
where ${\go m^{\prime}}$ is stated for $f_{\go m}(x)$. It is clear that the metric space $\{\go F,\go N\}$ is isometric to a subspace of Hilbert space. 

The transformation $\go m \rightarrow \go m^{\prime}$ from $\go P$ to $\go F$ is not one-to-one\footnote{Two random sets $\mathbb{A}$ taking two values $A$ and $\Omega \setminus A$  with equal probabilities $1/2$ and $\mathbb B$ taking values $\Omega$ and $\emptyset$ with the same probabilities have the same image $\go m^{\prime}$.}. However, $\go m^{\prime}$ contains ``essential information" on the measure $\go m$. Therefore, natural first step in the study of random sets distributions consists in investigation of properties of the space $\go F$ .

Together with the distance (2.3) we consider
\begin{eqnarray}\label{eq7}
{\go N}_p(\go m^{\prime}, \go n^{\prime}) = \Bigl(\int_{\Omega}\bigl|f_{\go m}(x)-f_{\go n}(x)\bigr|^p dm(x)\Bigr)^{1/p}, \;\; p\geq 1,\\
\nonumber
{\go N}_{\infty}(\go m^{\prime}, \go n^{\prime}) = \sup_{x \in \Omega}\bigl|f_{\go m}(x)-f_{\go n}(x)\bigr|=\lim_{p \to \infty}{\go N}_p(\go m^{\prime}, \go n^{\prime}). 
\end{eqnarray}

It is clear that ${\go N}^2(\go m^{\prime},\go n^{\prime})={\go N}_2^2(\go m^{\prime},\go n^{\prime})$ is negative definite kernel and, therefore, 
\begin{equation}\label{eq8}
\go K(\go m^{\prime},\go n^{\prime})=\int_{\Omega}f_{\go m}(x)f_{\go n}(x)dm(x)
\end{equation}
is a positive definite kernel on $\go F$. One of operations that may be defined on $\go F$ is multiplication:
\begin{equation}\label{eq9}
{\go m^{\prime}}\circ {\go n^{\prime}} =f_{\go m}(x)f_{\go n}(x).
\end{equation}
For non-random sets this operation corresponds to theirs intersection. With the operation (\ref{eq9}) and the kernel (\ref{eq8}) $\go F$ is a semigroup\footnote{It is easy to verify that ${\go m^{\prime}}\circ {\go n^{\prime}} \in \go F$ for all $\go m^{\prime}, \go n^{\prime} \in \go F$.} with positive definite kernel in sense of \cite{Kl}. 

There are other ways to turn $\go F$ into semigroup with positive definite kernel. For example, we may define new operation in $\go F$ as
\begin{equation}\label{eq10}
\go m^{\prime}*\go n^{\prime} =\bigl(1-f_{\go m}(x)\bigr)\cdot \bigl(1-f_{\go n}(x)\bigr)
\end{equation}
with the kernel
\begin{equation}\label{eq11} 
\go K^*(\go m^{\prime},\go n^{\prime})=\int_{\Omega}\bigl(1-f_{\go m}(x)\bigr)\cdot \bigl(1-f_{\go n}(x)\bigr)dm(x).
\end{equation}

\section{Operation $\circ$ for the case of discrete random sets}\label{sec3}
\setcounter{equation}{0}

As has been mentioned above, the operation $\circ$ between non-random sets corresponds to intersection of corresponding events. For random sets the values of this ``product" $\circ$ consist of the class of corresponding intersections. 

Let us turn to obtaining limit theorems connected to the large number of $\circ$-"multipliers. 

\begin{thm}\label{th1}
Let $\mathbb A$ be a discrete random set taking values $A_1,A_2, \ldots , A_n, \ldots$ with probabilities $p_1,p_2, \ldots ,p_n, \ldots$. Suppose that $p_1>0$ and $A_1 = \bigcap_{j=1}^{\infty}A_j$. Then 
\begin{equation}\label{eq12}
\lim_{n \to \infty} \Bigl(\sum_{k=1}^{\infty}\1_{A_k}(x)p_k\Bigr)^n = \1_{A_1}(x).
\end{equation}
\end{thm}
\begin{proof} We have
\[\Bigl(\sum_{k=1}^{\infty}\1_{A_k(x)}(x)p_k\Bigr)^n =  \]
\[=\1_{A_1}(x) \Big[p_1^n+\sum_{s=1}^{n-1}p_1^{k-s}\bigl(\sum_{j=2}^{\infty}p_j \bigr)^s {\binom{n}{s}} \Bigr]+\Bigl(\sum_{j=2}^{\infty}\1_{A_j}(x)p_j\Bigr)^n =\]
\[ = \bigl(1-(1-p_1)^n\bigr)\1_{A_1}(x)+\Bigl(\sum_{j=2}^{\infty}\1_{A_j}(x)p_j\Bigr)^n  \stackrel[n\to \infty]{}{\longrightarrow} \1_{A_1}(x).\]
Here we used the properties $A_1 \cap A_j = A_1$ and $\1_{A_1 \cap A_j}(x) =\1_{A_1}(x)\cdot \1_{A_j}(x)$ for all $j=1,2,\ldots$.
\end{proof}

As a particular case we obtain that if $A_1=\emptyset$ then the limit in (\ref{eq12}) is zero.

Let us look at ``classical" analogue of Theorem \ref{th1}. {\it Consider discrete random variable $X$ taking values $x_1, \ldots ,x_s, \ldots$ with positive probabilities $p_1, \ldots, p_s, \ldots$, where $x_1 = \min\{x_1, \ldots ,x_s,\ldots\}$. Suppose that $X_1, \ldots ,X_n$ are independent identically distributed (i.i.d) with $X$ random variables. Denote $X_{1:n}=\min\{X_1, \ldots ,X_n\}$. Then 
\[X_{1:n} \to x_1 \; \; \text{as}\; \; n\to \infty .\]
Here we have convergence in distribution}. To obtain this statement from Theorem \ref{th1} it is sufficient to apply it to sets $A_j = (-\infty,x_j]$.

For the case $A_1 = \emptyset$ the limit in (\ref{eq12}) is zero. It is interesting to study what is the speed of convergence to zero in this case. To understand the statement of the problem more precisely let us consider an example.

\begin{ex}\label{ex1} Let $A$ be a set from $\mathcal E$. Consider random set ${\mathbb A}$ taking two values $A$ and $\bar{A}=\Omega \setminus A$ with probabilities $1/2$ each. Then
\[ \Bigl(\1_A(x)\cdot \frac{1}{2}+\1_{\bar{A}}(x)\cdot \frac{1}{2}\Bigr)^n =\frac{1}{2^{n-1}}\Bigl(\1_A(x)\cdot \frac{1}{2}+\1_{\bar{A}}(x)\cdot \frac{1}{2}\Bigr).\]
Clearly, $A \cap \bar{A} = \emptyset$. Although for $n=1$ the probability of $\emptyset$ is zero (and we cannot apply Theorem (\ref{th1})), we have convergence to zero. However, if we introduce a ``normalization operator" to change the probabilities in $2^{n-1}$ times, we will obtain the same distribution of random set as at the first step. 
\end{ex}

Let $\mathbb A$ be a discrete random set taking value $A_1=\emptyset$ with probability $p_1>0$, and
\[f_{\mathbb A}(x) = \sum_{j=1}^{\infty}\1_{A_j}(x)p_j .\]
Then random set $\mathbb B$ taking values $A_j$, $j=2, \ldots ,j_n, \ldots$ with probabilities $p_j/(1-p_1)$ has function
\begin{equation}\label{eq13}
f_{\mathbb B}(x) = \frac{1}{1-p_1}f_{\mathbb A}(x).
\end{equation}

\begin{df}\label{de2}
Suppose that $\mathbb{A}$ is a discrete random set taking values $A_j$ with probabilities $p_j$, $j=1,2, \ldots $. We call $\mathbb{A}$ stable random set if for any integer $n \geq 2$ there exists positive number $\kappa_n$ such that
\begin{equation}\label{eq14}
f_{\mathbb{A}}^n(x) = \kappa_n f_{\mathbb{A}}(x).
\end{equation}
\end{df}

It is clear that:
\begin{enumerate} 
\item[i.] If $\mathbb{A}$ takes only one value $A$ with probability $1$ then $\mathbb{A}$ is stable random set.
\item[ii.] Random set from Example \ref{ex1} is stable.
\end{enumerate}

\begin{ex}\label{ex2} 
Suppose that a random set $\mathbb{A}$ takes non-empty values $A_1, \ldots , A_k$ with equal probabilities $1/k$. We also suppose that $A_i \cap A_j = \emptyset$ for $i \neq j$. Then $\mathbb{A}$ is stable. Really, we have
\[ \Bigl(\frac{1}{k} \sum_{j=1}^{k}\1_{A_j}(x)\Bigr)^n = \frac{1}{k^n}\sum_{j=1}^{k}\1_{A_j}(x), \]
and (\ref{eq14}) is true with $\kappa_n = k^{n-1}$.
\end{ex}

We can obtain some results on convergence to stable random sets. Theorem \ref{th1} gives sufficient conditions for the convergence to degenerate random sets. Below we propose another limit theorem.

\begin{thm}\label{th2} 
Suppose that $\mathbb{A}$ is a discrete random set and $\mathbb{B}$ is a stable random set with ``normalizing constant" $\kappa_n$. Suppose that 
\begin{equation}\label{eq15}
f_{\mathbb{A}}(x) =p_1 f_{\mathbb{B}}(x)+p_2 h(x),
\end{equation}
where $\kappa_n p_2^{n}/p_1^n \to 0$ as $n \to \infty$ and $f_{\mathbb{B}}(x)\cdot h(x) =0$. Then there exists a sequence $\lambda_n$ of positive constants such that
\begin{equation}\label{eq16}
\lambda_n \Bigl(f_{\mathbb{A}}(x)\Bigr)^n \longrightarrow f_{\mathbb{B}}(x)\;\; \text{as}\;\; n \to \infty . 
\end{equation}
\end{thm}
\begin{proof} Using the fact $f_{\mathbb{B}}(x)\cdot h(x) =0$ it is not difficult to calculate that 
\[ \lambda_n f_{\mathbb{A}}^{n}(x) - f_{\mathbb{B}}(x)=\Bigl(\frac{\lambda_n p_1^n}{\kappa_n}-1\Bigr)f_{\mathbb{B}}(x) +\lambda_n \cdot p_2^n \cdot h^n(x) \]
for any $\lambda_n >0$. Now it is sufficient to choose $\lambda_n$ so that $\lambda_n \cdot p_1^n /\kappa_n \to 1$ from below as $n \to \infty$. The convergence from below is needed to verify that $\lambda_n \cdot p_1^n /\kappa_n f_{\mathbb{B}(x)}$ corresponds to a random set.
\end{proof}

It is possible to use other definition of stability for the case of discrete random sets. Its idea is similar to that of the definition of casual stability given in \cite{KSKT}. 

\begin{df}\label{de3} Let $\mathbb{A}_a$ be a family of discrete random sets ($a>0$ is a parameter), $\p\{\mathbb{A}_a=A_j\}=p_j(a)$, $(j=1,2, \ldots)$. We say $\mathbb{A}_a$ is stable with respect to a family of transformation $a \rightarrow \xi_n(a)$ $n=1,2, \ldots$ if for any positive integer $n$
\begin{equation}\label{eq17}
 f_a^n(x) = f_{\xi_n(a)}(x), \;\; x\in \Omega,
\end{equation}
where $f_a(x) = \sum_{j=1}^{\infty}\1_{A_j}(x)p_j(a)$.
\end{df}

\begin{thm}\label{th3}
Let $\mathbb{A}_a$ be a family of discrete random sets ($a>0$ is a parameter), $\p\{\mathbb{A}_a=A_j\}=p_j(a) =a\cdot (1-a)^{j-1}$, $(j=1,2, \ldots)$, where $A_j \neq A_k$ for $j \neq k$ and
\[A_1 \subset A_2 \subset A_3 \subset \ldots.\]
Denote $\xi_n(a) = 1-(1-a)^n$. Then $\mathbb{A}_a$ is stable with respect to a family of transformation $\xi_n(a)$.  
\end{thm}
\begin{proof}
Let
\[f_a(x) = \sum_{j=1}^{\infty}\1_{A_j}(x)p_j(a) \]
for
\[A_1 \subset A_2 \subset A_3 \subset \ldots.\]
It is not difficult to calculate that
\[ f_a^n(x) =\sum_{k=0}^{\infty}\bigl(s_k^n(a)-s_{k+1}^n(a)\bigr)\1_{A_{k+1}}(x), \]
where $s_k(a) = \sum_{j=k+1}^{\infty}p_j(a)$. For the case $p_j(a) = a\cdot (1-a)^{j-1}$ we have
\[ s_k(a) = (1-a)^k, \;\; k=0,1,2, \ldots \]
It is obvious that 
\[ s_k (\xi_n(a)) =(1-a)^{n\cdot k}.\]
Therefore,
\[ f_a^n(x) = f_{\xi_n(a)}(x). \]
\end{proof}

\section{Operation $*$ for the case of discrete random sets}\label{sec4}
\setcounter{equation}{0}

The results for operation $*$ are very similar to that obtained for the case of $\circ$-operation. 

For the case of nonrandom sets the $\circ$-operation corresponds to intersection of sets and, therefore, to the product of corresponding functions $f_{\go{m}}(x)$ and $f_{\go{n}}(x)$ in general case. Operation $*$ corresponds to union of nonrandom sets and, therefore, to the intersection of their complements. This means that the study of $*$-operation is equivalent to that of $\circ$ for complements. For random case we only have to change $f_{\go{m}}(x)$ by $1-f_{\go{m}}(x)$.
We assume that the reader can formulate the corresponding results himself.

\section{Statistical testing}\label{sec5}
\setcounter{equation}{0}

Let us now consider a problem of statistical testing for random sets. Suppose that we have two random samples (that is $n$ i.i.d. random sets) $\mathbb{A}_1, \ldots , \mathbb{A}_n$ and $\mathbb{B}_1, \ldots , \mathbb{B}_n$, $n \geq 2$. We have to test the hypothesis $\go{m}^{\prime} = \go{n}^{\prime}$, where $\go{m}$ and $\go{n}$ are the distributions of $\mathbb{A}$ and $\mathbb{B}$ correspondingly. Testing of this hypothesis is equivalent to that of 
\begin{equation}\label{eq18}
{\go N}^2(\go m^{\prime}, \go n^{\prime}) = \int_{\Omega}\bigl(f_{\go m}(x)-f_{\go n}(x)\bigr)^2 dm(x) = 0,
\end{equation}
which is, in its turn, equivalent to 
\begin{equation}\label{eq19}
\int_{\Omega}\bigl(f^2_{\go m}(x)-f_{\go m}(x)\cdot f_{\go n}(x) \bigr)dm(x) = \int_{\Omega}\bigl(f_{\go m}(x)\cdot f_{\go n}(x) -f^2_{\go n}(x)\bigr)dm(x).
\end{equation}
To construct statistical test we have to replace $f_{\go m}(x)$ and  $f_{\go n}(x)$ by their empirical analogues. For such replacing in (\ref{eq19}) we shall have the sample from one-dimensional distributions. After that we may apply any free-of-distribution two-sample one-dimensional test. For the case of (\ref{eq18}) it is possible to use permutation test. 

Let us consider such procedures in more details. Suppose that the realizations of random sets $\mathbb{A}_j$ and $\mathbb{B}_j$ $(j=1, \ldots, n)$ are the sets $A_j$ and $B_j$ correspondingly. Then, from (\ref{eq5})
\[\mathcal{N}(\go{m}_n,\go{n}_n) = \frac{2}{n^2}\sum_{i=1,j=1}^{n}m(A_i\Delta B_j) -\frac{1}{n^2}\sum_{i=1,j=1}^{n}m(A_i\Delta A_j) - \frac{1}{n^2}\sum_{i=1,j=1}^{n}m(B_i\Delta B_j),  \] 
where $\go{m}_n$ and $\go{n}_n$ are empirical measures of the samples; $A_i$ are the results of observations from $\go{m}$ and $B_j$ are that from $\go{n}$. The test for identical distribution of $\go{m}^{\prime}$ and $\go{n}^{\prime}$ may be produced as a comparison of $\mathcal{N}(\go{m}_n,\go{n}_n)$ with the values of $\mathcal{N}$ statistic for permuted samples.

For applied problems the situation is a little bit more difficult. The realizations $A_j$ and $B_k$ ($j,k =1, \ldots ,n$) are not observable. Usually, statistician has a system of sets $C_1, \ldots , C_s$ such that 
\[ \bigcup_{j=1}^{s} C_j =\Omega,\quad \text{and}\quad C_i\cap C_j =\emptyset \quad\text{for}\quad i\neq j    \] 
and can only observe the fact of intersection of events $A_j$ and $B_k$ with $C_{l}$.\footnote{The system $\1_{C_1}(x), \ldots , \1_{C_s}(x)$ is an orthogonal system.} This means that instead of each $A_j$ statistician has $s$-dimensional vector $\mathbf{a}_j=(a_{j,1}, \ldots, a_{j,s})$, where $a_{j,l} = 1$ if the intersection $A_j \cap C_l \neq \emptyset$ and
$a_{j,l}=0$ in the opposite case. Of course, instead of $B_k$ the statistician has vector
$\mathbf{b}_k$ obtained in similar way. This means that instead of the problem two-sample test constructing for random sets $\mathbb{A}$ and $\mathbb{B}$ the statistician has to construct two-sample test for $s$-dimensional random vectors. This problem had been studied in \cite{Kl} on the basis of $\go{N}$ distance constructed by mean of ``ordinary" negative definite kernels. Therefore, we have essential connection between distances for random sets and metrics for random vectors. Let us note that the mean values 
\[ \bar{\mathbf{a}}_{n}=\frac{1}{n}\sum_{j=1}^{n}\mathbf{a}_j \quad \text{and} \quad \bar{\mathbf{b}}_{n}=\frac{1}{n}\sum_{j=1}^{n}\mathbf{b}_j \]
give consistent statistical estimators of probabilities $m(\mathbb{A} \cap C_l)$ and $m(\mathbb{B} \cap C_l)$ ($l=1, \ldots s$). In other words, these are estimators for the best approximations of $f_{\mathbb{A}}(x)$ and $f_{\mathbb{B}}(x)$ by elements of orthogonal system $\1_{C_1}(x), \ldots , \1_{C_s}(x)$. Such approximations are similar to (\ref{eqA}), but for random sets instead of non-random in (\ref{eqA}). Really, in above described scheme statistician tests not equality of distributions of $\mathbb{A}$ and $\mathbb{B}$ but the identity of their best approximations by elements of orthogonal system $\1_{C_1}(x), \ldots , \1_{C_s}(x)$.

\section{Acknowledgment}
The work was partially supported by Grant GACR 16-03708S.


\begin{thebibliography}{99}

\bibitem{Ma}
Georges Materon (1975).
\newblock Random Sets and Integral Geometry.
\newblock John Wiley \& Sons, New York, London, Sydney, Toronto.

\bibitem{Mo}
Ilya Molchanov (2005).
\newblock Theory of Random Sets.
\newblock Springer-Verlag, London.

\bibitem{Vl} 
D.A. Vladimirov (2002).
\newblock Boolean Algebras in Analysis.
\newblock Springer, Science+Business Media, Dordrecht. 

\bibitem{Kl}
Lev B. Klebanov (2005).
\newblock $\go N$-Distances and their Applications.
\newblock The Karolinum Press, Prague.

\bibitem{KSKT}
Lev B Klebanov, Lenka Sl\'{a}mov\'{a}, Ashot Kakosyan, Gregory Temnov (2014).
\newblock Casual Stability of Some Systems of Random Variables,
\newblock arXiv: 1408.3864, v1, [math.PR], 1-13.

\end{thebibliography}
\end{document}